\documentclass[11pt]{amsart}
\usepackage{amsfonts,amssymb,amscd,amsmath,enumerate,verbatim,calc,latexsym,
pstcol,pst-plot,pst-3d,mathrsfs}

\def\NZQ{\mathbb}               
\def\NN{{\NZQ N}}

\def\ZZ{{\NZQ Z}}

\def\frk{\mathfrak}               

\def\mm{{\frk m}}

\def\Phi{{\frk N}}
\def\opn#1#2{\def#1{\operatorname{#2}}} 
\opn\chara{char} \opn\length{\ell} \opn\pd{pd} \opn\rk{rk}
\opn\projdim{proj\,dim} \opn\injdim{inj\,dim} \opn\rank{rank}
\opn\depth{depth} \opn\grade{grade} \opn\height{height}
\opn\embdim{emb\,dim} \opn\codim{codim}

\opn\Tr{Tr} \opn\bigrank{big\,rank}
\opn\superheight{superheight}\opn\lcm{lcm}
\opn\trdeg{tr\,deg}
\opn\reg{reg} \opn\lreg{lreg} \opn\ini{in} \opn\lpd{lpd}
\opn\size{size}\opn{\mult}{mult}\opn\sdepth{sdepth}
\opn\div{div} \opn\Div{Div} \opn\cl{cl} \opn\Cl{Cl}
\opn\Spec{Spec} \opn\Supp{Supp} \opn\supp{supp} \opn\Sing{Sing}
\opn\Ass{Ass} \opn\Min{Min}
\opn\Ann{Ann} \opn\Rad{Rad} \opn\Soc{Soc}
\opn\Syz{Syz} \opn\Im{Im} \opn\Ker{Ker} \opn\Coker{Coker}
\opn\Am{Am} \opn\Hom{Hom} \opn\Tor{Tor} \opn\Ext{Ext}
\opn\End{End} \opn\Aut{Aut} \opn\id{id} \opn\ini{in}

\opn\nat{nat}
\opn\pff{pf}
\opn\Pf{Pf} \opn\GL{GL} \opn\SL{SL} \opn\mod{mod} \opn\ord{ord}
\opn\Gin{Gin}
\opn\Hilb{Hilb}\opn\adeg{adeg}\opn\std{std}\opn\ip{infpt}
\opn\Pol{Pol}
\opn\sat{sat}
\opn\Var{Var}
\opn\Gen{Gen}
\def\ra{{\rightarrow}}
\opn\aff{aff} \opn\con{conv} \opn\relint{relint} \opn\st{st}
\opn\lk{lk} \opn\cn{cn} \opn\core{core} \opn\vol{vol}
\opn\link{link} \opn\star{star}
\opn\gr{gr}

\def\Fc{{\mathcal F}}
\def\pot#1#2{#1[\kern-0.28ex[#2]\kern-0.28ex]}

\opn\dirlim{\underrightarrow{\lim}}
\opn\inivlim{\underleftarrow{\lim}}
\let\union=\cup
\let\sect=\cap
\let\dirsum=\oplus

\let\iso=\cong

\let\Dirsum=\bigoplus

\let\sus=\subseteq

\let\to=\rightarrow

\def\Implies{\ifmmode\Longrightarrow \else
        \unskip${}\Longrightarrow{}$\ignorespaces\fi}
\def\implies{\ifmmode\Rightarrow \else
        \unskip${}\Rightarrow{}$\ignorespaces\fi}
\def\iff{\ifmmode\Longleftrightarrow \else
        \unskip${}\Longleftrightarrow{}$\ignorespaces\fi}

\let\:=\colon
\let\sus=\subseteq
\newcommand{\Pscr}{\mathscr P}
\newcommand{\bfa}{{\bf a}}
\newcommand{\mto}[1]{\stackrel{#1}\longrightarrow}
\newtheorem{Theorem}{Theorem}[section]
\newtheorem{Lemma}[Theorem]{Lemma}
\newtheorem{Corollary}[Theorem]{Corollary}
\newtheorem{Proposition}[Theorem]{Proposition}

\theoremstyle{definition}
\newtheorem{Definition}[Theorem]{Definition}

\theoremstyle{remark}
\newtheorem{Remark}[Theorem]{Remark}

\newtheorem{Examples}[Theorem]{Examples}

\let\epsilon\varepsilon
\let\phi=\varphi
\let\kappa=\varkappa
\def\qed{\ifhmode\textqed\fi
      \ifmmode\ifinner\quad\qedsymbol\else\dispqed\fi\fi}
\def\textqed{\unskip\nobreak\penalty50
       \hskip2em\hbox{}\nobreak\hfil\qedsymbol
       \parfillskip=0pt \finalhyphendemerits=0}
\def\dispqed{\rlap{\qquad\qedsymbol}}

\opn\dis{dis}\opn\Mon{Mon}
\def\pnt{{\raise0.5mm\hbox{\large\bf.}}}
\def\lpnt{{\hbox{\large\bf.}}}
\opn\Lex{Lex}

\newcommand{\inD}[1][\relax]{\def\argone{#1}\def\temprelax{\relax}
  \ifx\argone\temprelax\right.\else\,\middle|#1\right.{}\fi}

\newif\ifbinary
\binarytrue

\begin{document}

\title{Gr\"obner bases of syzygies and Stanley depth}

\author{Gunnar Fl{\o}ystad and J\"urgen Herzog}

\address{J\"urgen Herzog, Fachbereich Mathematik, Universit\"at Duisburg-Essen, Campus Essen, 45117
Essen, Germany} \email{juergen.herzog@uni-essen.de}

\address{Gunnar Fl{\o}ystad,   Department of Mathematics, Johs.\ Brunsgt.\ 12, 5008 Bergen, Norway} \email{gunnar@mi.uib.no}

\begin{abstract}
Let $F_{\lpnt}$ be a any free resolution of a ${\ZZ^n}$-graded submodule
of a free module over the polynomial ring $K[x_1, \ldots, x_n]$.
We show that for a suitable term order on $F_{\lpnt}$, the initial module
of the $p$'th syzygy module $Z_p$ is generated by terms $m_ie_i$
where the $m_i$ are monomials in $K[x_{p+1}, \ldots, x_n]$.
Also for a large class of free resolutions $F_{\lpnt}$,
encompassing Eliahou-Kervaire resolutions,
we show that a Gr\"obner basis for $Z_p$ is given by the
boundaries of generators of $F_p$.

We apply the above to give lower bounds for the Stanley depth of the syzygy
modules $Z_p$. We also show that if $I$ is any squarefree ideal in
$K[x_1, \ldots, x_n]$, the Stanley depth of $I$ is  at least
of order $\sqrt{2n}$.
\end{abstract}



\maketitle

\section*{Introduction}
Let $K$ be a field and $S=K[x_1,\ldots,x_n]$ the polynomial ring in  $n$
variables over $K$. We study Gr\"obner bases of syzygies of finitely generated
$\ZZ^n$-graded modules  over $S$, and apply this to give lower bounds for the
Stanley depth of syzygy modules.

Fix any monomial order $<$ on $S$ and let $F$ be a free $\ZZ^n$-graded
$S$-module with a homogeneous basis
$\mathcal{F}=e_1,\ldots,e_m$. We define a monomial order on $F$ by
setting $ue_i>ve_j$ if $i<j$, or $i=j$ and $u>v$,
where $u$ and $v$ are monomials of $S$.
If $M$ is a $\ZZ^n$-graded submodule of $F$,
a basic observation is that the initial module
$\ini(M)$ does not depend on the monomial order $<$ on $S$ but only on
the basis $\mathcal{F}$.
Therefore we denote the initial module of $M$ with respect to
this monomial order by $\ini_{\mathcal F}(M)$.
We have  $\ini_{\mathcal F}(M)=\Dirsum_{j=1}^mI_je_j$, where each $I_j$ is a
monomial ideal.

We call the basis $\Fc$ of $F$ {\em lex-refined},
if $\deg(e_1)\geq \deg(e_2)\geq \ldots\geq  \deg(e_m)$ in the lexicographical
order.
Our first main result, Theorem \ref{main}, shows that the initial modules
of syzygy modules, when choosing a lex-refined basis, have a simple and
natural property~: let $M$ be a $\ZZ^n$-graded submodule of a free module
$F_0$ with free resolution  $\cdots \to F_2\to F_1\to M\to 0$.
For  $0 \leq p\leq n$ let $Z_p \subseteq F_p$ be the $p$'th syzygy module.
Then the initial module $\ini_{\mathcal F}(Z_p)$ is $\Dirsum_{j=1}^mI_je_j$,
where the minimal
set of monomial generators  of each $I_j$ belongs to $K[x_{p+1},\ldots,x_n]$.

This theorem may remind the reader to a well-known result of F.-O. Schreyer, see Section 5.5 of
\cite{Ei}, who
showed that for any finitely generated module $M$ one can find a free resolution
and suitable monomial orders on the free modules of the resolution such that the
initial modules of the syzygies enjoy the same nice property as described above.
The point here is that no assumption is made on the
$\ZZ^n$-graded  resolution on $M$. In particular, the theorem is valid for the
 graded minimal free resolution of $M$.

In general of course it is not so easy to compute the initial module of a syzygy
module in a free resolution $F_{\lpnt}$ of a module $M$. But for certain
classes of resolutions this may be done in a pleasant way.
We say that the resolution has {\em boundary Gr\"obner bases}
if for each $p$  there exists a
basis $\mathcal{F}_p$ of $F_p$ such that
$\ini_{{\mathcal F_p}}(Z_p(F_{\lpnt}))$
is generated by the initial terms of $\partial_{p+1}(e_i)$
where $\partial_{\lpnt}$ denotes the differential of $F_\lpnt$ and $e_i$ ranges
over $\Fc_{p+1}$.
If $F_{\lpnt}$ has such bases,  the initial modules of the syzygies can easily be read
off from the matrices describing $\partial_{\lpnt}$ with respect to these bases.
We show that the Taylor resolution as well as the Eliahou--Kervaire
resolution have boundary Gr\"obner bases.

\medskip
We then apply the first result on syzygies to give lower bounds for the
Stanley depth of syzygies. A {\em Stanley decomposition} of a $\ZZ^n$-graded
$S$-module $M$ is a direct sum decomposition $M=\Dirsum_{i=1}^mu_iK[Z_i]$
of $M$  as a $\ZZ^n$-graded $K$-vector space, where each $u_i$ is a homogeneous
element of $M$, $K[Z_i]$ is a polynomial ring is a set of variables
$Z_i\subset \{x_1,\ldots,x_n\}$, and each $u_iK[Z_i]$ is a free $K[Z_i]$-submodule
of $M$. The minimum of the numbers $|Z_i|$ is called the Stanley depth of this
decomposition. The {\em Stanley depth} of $M$, denoted $\sdepth M$,  is the
maximal Stanley depth of a Stanley decomposition of $M$. In his paper \cite{Sta}
Stanley conjectured that $\sdepth M\geq \depth M$. This conjecture is widely open.
In the papers  listed in the references in this paper and the references therein,
the reader can inform himself you about the present status of the conjecture.

Naively one could expect, that like for the ordinary depth, the Stanley depth of
the first syzygy module $Z_1(M)$ of a $\ZZ^n$-graded module $M$ 
is one more than that of $M$,
as along as $M$  is not free. This of course would immediately imply Stanley's
conjecture. However this is not the case. For example, if we let
 $\mm=(x_1,\ldots,x_n)$ be the graded maximal ideal of $S$. Then
$\sdepth S/\mm=0$, while $\sdepth \mm=\lceil n/2\rceil$, as shown in \cite{Biro}.
Nevertheless it might be true that one always has $\sdepth Z_1(M)\geq \sdepth M$.
But at the moment even the inequality $\sdepth I\geq \sdepth S/I$ for a monomial
ideal $I$ is unknown. However as one of the main results of this paper, we show 
in Theorem~\ref{MainSyz} that if $M$ is a $\ZZ^n$-graded free submodule of free
$S$-module, then for $p<n$  the $p$'th  syzygy module of $M$ with respect to any
(not necessarily minimal) $\ZZ^n$-graded free resolution of $M$  has Stanley depth
at least $p+1$.

One problem in proving such a result as stated in  Theorem~\ref{MainSyz} is the
fact that at present no method is known to compute the Stanley depth of a
$\ZZ^n$-graded module in a finite number of steps. So far this can be done  only
for modules of the form $I/J$ where $J\subset I\subset S$ are monomial ideals, see
 \cite{He}. It seems not even to be  known that for a monomial ideal $I\subset S$
 one has $\sdepth I\dirsum S=\sdepth I$, as one would expect. The only method
known to get a lower bound for the Stanley depth of a $\ZZ^n$-graded module $M$ is
 to find a suitable filtration of the module whose factors are of the form $I$ or
$S/I$ where $I$ is a monomial ideal. The Stanley depth of $M$ is then just the
minimum of the Stanley depth of the factors of the filtration.
This enables us to give lower bounds for the Stanley depth of a syzygy module by
using that if  the initial module $\ini_{\mathcal F}(Z_p)$
is $\Dirsum_{j=1}^mI_je_j$, then  the monomial
ideals $I_j$ are the factors of a suitable
filtration of $Z_p$. Therefore the Stanley depth of $M$ is
greater or equal to the minimum of the Stanley depths of the $I_j$.

We also apply the second result on Gr\"obner basis of syzygies to show
that when $I\subset S$ is a monomial complete intersection minimally generated
by $m$ elements,  the $p$'th syzygy module of $S/I$ has Stanley depth at least
$n-\lfloor \frac{m-p}{2}\rfloor$, see Proposition~\ref{regular}. This indicates
that our general lower bounds for the Stanley depth of syzygy modules are far from
being optimal.

\medskip Somewhat independently of the above, we show that if $I$ is a squarefree
ideal over the polynomial ring in $n$ variables, the
Stanley depth of $M$ is at least of order $2\sqrt{n}$, Theorem \ref{SqfreeStDe}.
This is quite in contrast to
ordinary monomial ideals. In fact it is known by Cimpoea\c s \cite{Ci1} that
sufficiently high powers of $\mm$ have Stanley depth 1. The proof of
Theorem~\ref{SqfreeStDe} is based on a construction
of interval partitions \cite{Ke} by M.Keller et. al. which is further refined
M.Ge et. al. in \cite{Lin}.
Applying this we give lower bounds, Theorem \ref{squarefree},
for the Stanley depth of syzygy modules
of a squarefree submodule of a free module. This bound is considerably better
than what we have for arbitrary $\ZZ^n$-graded submodules.

\medskip
The organisation of the paper is as follows. In Section 1 we consider resolutions
of $\ZZ^n$-graded $S$-modules and initial modules of their syzygy modules
determined by chosing ordered multihomogeneous bases for the terms in
the resolutions. We first prove that for lex-refined orders, the generators
of the initial module of the $p$'th syzygy module does not involve the first $p$
variables.
We then give classes of resolutions which have boundary Gr\"obner bases.
In Section 2 we give the lower bounds on Stanley depth of syzygies. These
are a consequence of the results in Section 1, and in the squarefree case, a consequence
of the result in the next Section 3. In this last section, we show that the Stanley
depth of any squarefree monomial ideal in $n$ variables is at least of order
$2\sqrt{n}$.

\section{Gr\"obner basis of syzygies of multigraded modules}
\label{1}

We consider term orders on $\ZZ^n$-graded free modules over a polynomial
ring in $n$ variables which are determined by fixing a multihomogeneous
basis $e_1, \ldots, e_m$ of the free module and comparing terms
$ue_i$ and $ve_j$ by first comparing their basis elements.
For such term orderings the initial term of any multihomogeneous element
will be determined solely by the ordering of the $e_i$'s, and so also the
initial module of any multihomogeneous submodule.

   A natural ordering of the $e_i$'s is by lexicographic ordering of their
multidegrees. We show then that the syzygies of a free resolution of a
$\ZZ^n$-graded module have the nice and natural property that for each
successive syzygy module we miss an extra variable in their generating set.

\medskip
Let $K$ be a field,  $S=K[x_1,\ldots,x_n]$ the polynomial ring over $K$ in $n$ indeterminates,
$M$  a  $\ZZ^n$- graded  $S$-module and
\begin{eqnarray}
\label{resolution}
F_{\lpnt}\: \cdots\to  F_p\to F_{p-1}\to \cdots \to F_1 \to  M\to 0
\end{eqnarray}
a $\ZZ^n$-graded (not necessarily minimal) free $S$-resolution of $M$ with all $F_i$ finitely generated. Let $Z_p\subset F_p$ be the $p$'th syzygy module of $M$ (with respect to this resolution).  We are interested in the initial module of the syzygy module $Z_p$ with respect to a monomial order on $F_p$. In this paper we restrict our attention to monomial orders of the following type: we denote by $\Mon(S)$ the set of monomials of $S$,  fix  a multihomogeneous basis ${\mathcal F}=e_1,\ldots,e_m$ of $F_p$ and  a monomial order $<$ on $S$,  and define a monomial order on $F_p$ by setting
\begin{eqnarray}
\label{order}
ue_i>ve_j,\quad \text{if $i<j$, or $i=j$ and $u>v$.}
\end{eqnarray}
Here $u,v\in\Mon(S)$. We denote by $\ini_<(Z_p)$ the monomial submodule of $F_p$ which is generated by all the initial monomials of elements of  $Z_p$. Notice that
\begin{eqnarray}
\label{dirsum}
\ini_<(Z_p)=\Dirsum_{j=1}^mI_je_j,
\end{eqnarray}
where each $I_j$ is a monomial ideal.

Since $Z_p$ is $\ZZ^n$-graded, $\ini_<(Z_p)$ is generated by  the initial monomials of  multihomogeneous elements of $Z_p$. Let $z=\sum_{i=1}^mf_ie_i$ be a multihomogeneous element in $Z_p$. Then each $f_i$ is a term $a_iu_i$ with $a_i\in K$ and $u_i$ a monomial. Thus $\ini_<(z)=u_je_j$, where $j$ is the smallest number such that $a_i\neq 0$. This consideration shows that for the above monomial order, the  initial module of a  syzygy module  of a monomial ideal {\it depends only on the given basis}  not on the chosen monomial  order on $S$. Thus we write $\ini_{\mathcal F}(Z_p)$ to denote the initial module of $Z_p$ with respect to this monomial order induced by $\mathcal F$.

In general,  for a given resolution $F_\lpnt$ there are several equally natural choices of multihomogeneous  bases for the $F_p$. Here we will choose  for each $F_p$ a basis  compatible with the lexicographical order of the multidegrees of the basis elements. We call such a basis of $F_p$ {\em lex-refined}. Thus a basis ${\mathcal F}=e_1,\ldots,e_m$ of $F_p$ of multihomogeneous elements is lex-refined, if $\deg(e_1)\geq \deg(e_2)\geq \ldots\geq  \deg(e_m)$ in the lexicographical order.

\medskip
\begin{Theorem}
\label{main} Suppose $M$ is a submodule of a free module $F_0$.
Let $0 \leq p\leq n$ be an integer, and let $\mathcal F$ be a lex-refined basis of $F_p$.
Then $\ini_{\mathcal F}(Z_p)=\Dirsum_{j=1}^mI_je_j$, where the minimal set of monomial generators  of each $I_j$ belongs to $K[x_{p+1},\ldots,x_n]$.
\end{Theorem}

This theorem is an immediate consequence of the following

\begin{Proposition}
\label{gunnar}
Let $\varphi\: F\to G$ be a homomorphism of finitely generated $\ZZ^n$-graded free $S$-modules with $M=\Im \varphi$ and $N=\Ker \varphi$.
Let ${\mathcal G}=e_1',\ldots,e_d'$ be a lex-refined basis of $G$ and ${\mathcal F}=e_1,\ldots, e_c$ a lex-refined basis of $F$, and assume that
for some $p\leq n$, $\ini_{\mathcal G}(M)$ is generated by monomials not divisible by $x_1,\ldots, x_{p-1}$. Then $\ini_{\mathcal F}(N)$ is generated by monomials not divisible by $x_1,\ldots,x_p$.
\end{Proposition}

\begin{proof}
1. Let $s\in N$ be a multihomogeneous  element, and let $x_r$ be the variable with least index such that $x_r$ divides $\ini(s)$. To demonstrate the desired property of $\ini_{\mathcal F}(N)$, it will be sufficient to show that there exists an element $\tilde{s}\in N$ such that
\begin{enumerate}
\item[(1)] $\ini(s)=\ini(\tilde{s})$;
\item[(2)] $x_r$ divides  $\tilde{s}$, if $r\leq p$.
\end{enumerate}
Indeed, if (1) and (2) are satisfied, then $t=\tilde{s}/x_r\in  N$ and $\ini(s)=x_r\ini(t)$. Thus we may delete $\ini(s)$ as a generator of $\ini_{\mathcal F}(N)$.

\noindent 2. In order to show the existence of $\tilde{s}$ with these properties, we write $s=s'+s''$, where  $s'$ is the sum of all terms of $s$ which are not  divisible by one of the variables $x_1,\ldots,x_{r-1}$, and $s''$ the sum of the other terms in $s$.

Let $s=\sum_{i=1}^ca_iu_ie_i$ with $a_i\in K$ and $u_i\in \Mon(S)$. For simplicity we may assume that $\ini(s)=u_1e_1$. Then, since $s$ is multihomogeneous and since $\mathcal F$ is a lex-refined basis of $F$, it follows that $u_i\leq  u_j$ in the lexicographical order for all $i\leq j$ in the support of $s$. Since $\ini(s')=\ini(s)=u_1e_1$, it follows that $u_1\leq u_i$ for all $i\in \supp(s')$. The monomial $u_1$ is divisible by $x_r$, and no $u_i$ with $i\in \supp(s')$ is divisible by any $x_j$ for $j<r$. Hence the inequality $u_1\leq u_i$ implies that $x_r$ divides all $u_i$ with $i\in \supp(s')$. In other words, $x_r$ divides $s'$.

\noindent 3. Since $s\in N$, it follows that $\varphi(s')=-\varphi(s'')$. We denote this element by $z$. Since $z=\varphi(s')$,  we see that $x_r$ divides $z$, and since $z=-\varphi(s'')$ it follows that each of the terms of $z$ is divisible by at least one $x_i$ with $i<r$.

Let  $\ini(z/x_r)= we_k'$.  Since $\deg (z/x_r)=\deg(s)-\epsilon_r$, where  $\epsilon_r$ is the $r$th canonical basis vector of $\ZZ^n$, it follows that  $\deg(e_k')_j\leq \deg(s)_j$ for $j=1,\ldots,r-1$, and since each term of $z/x_r$ is divisible at least one $x_j$ with $j<r$, we conclude   that $\deg(e_k')_j<\deg(s)_j$ for at least one $j<r$.

\noindent 4. Now we may write $z/x_r=\sum b_iv_ig_i$ with $b_i\in K$ and $v_i\in \Mon(S)$, where the $g_i$ form a reduced Gr\"obner basis of $M$ (with respect to the monomial order induced by $\mathcal G$) and with the  additional property that $\ini(v_ig_i)\leq \ini(z/x_r)=we_k'$ for all $i$. Let $\ini(g_i)= w_{j_i}e_{j_i}'$. Then $j_i\geq k$, and hence $\deg(e_{j_i}')\leq \deg(e_k')$ in the lexicographic order. Thus for all $i$ it follows that $\deg(e_{j_i}')_j\leq \deg(s)_j$ for $j=1,\ldots,r-1$ with strict inequality for at least one $j$.

\noindent 5. According to our hypothesis, we may assume that none of the variables $x_1,\ldots,x_{p-1}$ divides any of the $w_{j_i}$. Thus, together with what we have shown in the last paragraph we see   that for all $i$ we have that $\deg(g_i)_j\leq \deg(s)_j$ for $j=1,\ldots,r-1$,  but with strict inequality for at least one $j<r$. We may now lift each $g_i$ to an element $f_i$ of $F_p$ with the same multidegree, and set
\[
s'''=\sum b_iv_i f_i \quad \text{and} \quad \tilde{s}=s'-x_rs'''.
\]
Obviously,  $x_r$ divides $\tilde{s}$ and $\tilde{s}\in N$. We claim that $\ini(x_rs''')<\ini(s')$. Since $\ini(s)=\ini(s')$, the claim will then imply that  $\ini(\tilde{s})=\ini(s)$, as desired.

In order to prove the claim, assume to the contrary that $\ini(x_rs''')\geq \ini(s')=u_1e_1$. Since $\ini(x_rv_if_i)\geq \ini(x_rs''')$ for some $i$, it follows for this index $i$ that $\ini(x_rv_if_i)\geq u_1e_1$ which implies that  $\ini(f_i)=ve_1$ for some $v\in \Mon(S)$. In particular, $\deg(e_1)_j\leq \deg(f_i)_j$ for all $j$. Since $\deg(f_i)=\deg(g_i)$, and since there exists $j<r$ such that $\deg(g_i)_j<\deg(s)_j=\deg(u_1e_1)_j$, it follows that $\deg(e_1)_j<\deg(u_1e_1)_j$ for some $j<r$. This implies that $x_j$ divides $u_1$, a contradiction.
\end{proof}


\medskip
Let $F$ be a finitely generated $\ZZ^n$-graded $S$-module with multigraded basis ${\mathcal F}=e_1,\ldots,e_m$, and $M$ a $\ZZ^n$-graded submodule of $F$. In general it is not easy to compute $\ini_{\mathcal F}(M)$ explicitly. In the following we describe resolutions  where for suitable bases the initial modules of the syzygies can be simply determined. These bases are however not
necessarily lex-refined.

\begin{Definition}
Let $F_\lpnt$ be the resolution (\ref{1}) with differential $\partial_\lpnt$. It has {\em boundary Gr\"obner bases} if for each $F_p$ there exists a basis ${\mathcal F}_p$ such that
\[
\ini_{{\mathcal F}_p}(Z_p(F_\lpnt))=(\ini_{{\mathcal F}_p}(\partial_{p+1}(e_i))\: e_i\in\mathcal{F}_{p+1}).
\]
\end{Definition}

Resolutions with boundary Gr\"obner bases have the pleasant property that the initial modules of the syzygies can be immediately read off from the matrices describing the differential maps with respect to these bases.

The resolutions we have in mind arise as iterated mapping cones. So let $I\subset S$ be a monomial ideal with monomial generators $u_1,\ldots,u_m$. The iterated mapping cone resolution is constructed inductively by using induction on the number of generators. For $m=1$ it is just the complex $F_\lpnt^{(1)}\: 0\to S(-{\bf a}_1)\to S\to S/(u_1)\to 0$, where ${\bf a}_1$ is the multidegree of $u_1$ and the differential of the complex is multiplication by $u_1$.
Let $I_j$ be the ideal generated by $u_1, \ldots, u_j$, and
suppose for some $j<m$ we have already constructed the resolution $F_\lpnt^{(j)}$ of $S/I_j$. Consider the exact sequence
\[
0\to I_{j+1}/I_j\to S/I_j\to S/I_{j+1}\to 0.
\]
Observe that
\begin{equation} \label{GroEqGj}
I_{j+1}/I_j\iso (S/(I_j\: (u_{j+1}))(-{\bf a}_{j+1}),
\end{equation}
where ${\bf a}_{j+1}=\deg u_{j+1}$. Let $G^{(j)}_\lpnt$ be a $\ZZ^n$-graded free $S$-resolution of this cyclic module and $\varphi^{(j)}\:  G^{(j)}_\lpnt\to F_\lpnt^{(j)}$ a complex homomorphism of $\ZZ^n$-graded complexes extending the inclusion map $I_{j+1}/I_j\to S/I_j$. Then we define $F^{(j+1)}_\lpnt$ as the mapping cone of $\varphi^{(j)}$.

The free resolution obtained by iterated mapping cones is not at all unique. It depends on the choice of the free resolutions $G^{(j)}_\lpnt$ as well as on the complex homomorphisms $\varphi^{(j)}$.

Here are a few prominent examples of resolutions which arise as iterated mapping cones.
\begin{Examples}
\label{examples}
(a) The Taylor complex (cf.\ \cite{Ei}) is an iterated mapping cone. Let $u_1,\ldots,u_m$ be a sequence of monomials. Assuming the Taylor complex for any sequence of monomials of length $m-1$ is already constructed, one constructs the Taylor complex for $u_1,\ldots,u_m$ by choosing for
$F^{(m-1)}_\lpnt$ the Taylor complex of the sequence $u_1,\ldots, u_{m-1}$, and for  $G^{(m-1)}_\lpnt$ one  takes the Taylor complex  for the sequence
$$u_1/\gcd(u_1,u_m),\ldots,u_{m-1}/\gcd(u_{m-1},u_m)$$ which is a system of generators of $(u_1,\ldots,u_{m-1}):(u_m)$.
The map $\varphi_{m-1}$ can be defined in a canonical way.

The Taylor complex provides a $\ZZ^n$-graded free $S$-resolution of $S/(u_1,\ldots,u_m)$ (which in general is not minimal). In case $u_1,\ldots,u_m$ is a regular sequence, the Taylor complex coincides with the Koszul complex of this sequence and is minimal.

(b) A monomial ideal $I\subset S$ is said to have linear quotients, if $I$ is generated by homogeneous polynomials $u_1,\ldots,u_m$ with $\deg u_1\leq \deg u_2 \leq \cdots \leq \deg u_m$ such that each of the colon ideals $L_j=(u_1,\ldots,u_{j})\: (u_{j+1})$ is generated by a subset of the variables. For such an ideal  we can use in the construction of the iterated mapping cone for each $j$ the Koszul complex as $G^{(j)}_\lpnt$ to resolve $S/L_j$. Considering the degrees of the resolutions at each step we see that $\varphi_j(G^{(j)}_\lpnt)\subset \mm F^{(j)}_\lpnt$, so that in this case the iterated mapping cone provides a minimal free $\ZZ^n$-graded resolution of $S/I$.

An important special case is that of  a stable ideal. Recall that a monomial ideal $I$ is called stable if for all monomial $u\in I$ the monomial $x_j(u/x_{m(u)})\in I$  for all $j<m(u)$. Here $m(u)$ is the largest index with the property that $x_{m(u)}$ divides $u$. Let the minimal set of monomial generators $u_1,\ldots,u_m$ of $I$ be ordered in such a way that for $i<j$ either   $\deg u_i <\deg u_j$, or  $\deg u_i=\deg u_j$ and $u_j <u_i$ in the lexicographic order. Then with respect to this sequence of generators,  $I$ has linear quotients. The corresponding iterated mapping cone yields the so-called Eliahou-Kervaire resolution of $S/I$, provided the complex homomorphisms  $\varphi_j$ at each step are chosen properly.
\end{Examples}

The next lemma is a direct  consequence of the following two observations:

\medskip
\noindent
(i) Let $F$ be a finitely generated $\ZZ^n$-graded $S$-module with multigraded basis ${\mathcal F}=e_1,\ldots,e_m$, and $M$ a $\ZZ^n$-graded submodule of $F$ with  $\ini_{\mathcal F}(M)=\Dirsum_{j=1}^mI_je_j$. For $j=1,\ldots,m$ set $F(j)=\Dirsum_{i=j}^mSe_i$ and let $F(m+1)=0$. Then  for the  factors of the induced filtration
\begin{eqnarray}
\label{filtration}
M=M\sect F(1)\supset  M\sect F(2)\supset \cdots \supset M\sect F(m)\supset M\sect F(m+1)=(0)
\end{eqnarray}
we have
\begin{eqnarray}
\label{factors}
M\sect F(j)/ M\sect F(j+1)\iso I_j(-\deg(e_j)).
\end{eqnarray}

\medskip
\noindent
(ii) Let $M$ be a finitely generated $\ZZ^n$-graded module and $N\subset M$  a $\ZZ^n$-graded submodule of $M$, $F_\lpnt$ a $\ZZ^n$-graded free $S$-resolution of $M$,  $G_\lpnt$ a $\ZZ^n$-graded free $S$-resolution of $N$ and $\varphi_\lpnt\: G_\lpnt\to F_\lpnt$ a $\ZZ^n$-graded complex homomorphism which extends the inclusion map $N\to M$. The mapping cone $C_\lpnt$ of $\varphi_\lpnt$ is a $\ZZ^n$-graded free resolution of $M/N$. For the syzygies of these complexes we have for all $i\geq 0$ the following exact sequences
\begin{eqnarray}
\label{syzexact}
0\to Z_i(F_\lpnt)\to Z_i(C_\lpnt)\to Z_{i-1}(G_\lpnt)\to 0,
\end{eqnarray}
where we set $Z_{-1}(G_\lpnt)=N$.

\begin{Lemma}
\label{groebner}
With the notation introduced in (ii), let ${\mathcal G}=g_1,\ldots, g_r$ be a basis of $G_{i-1}$,  ${\mathcal F}=f_1,\ldots, f_s$ a basis of $F_i$ and ${\mathcal C}$ the basis of $C_i$ which is obtained by composing $\mathcal G$ with $\mathcal F$, that is, ${\mathcal C}=g_1,\ldots,g_r,f_1,\ldots,f_s$. Then
\[
\ini_{\mathcal{C}}Z_i(C_\lpnt)= \ini_{\mathcal G}Z_{i-1}(G_\lpnt)\dirsum \ini_{\mathcal F}Z_{i}(F_\lpnt). \]
\end{Lemma}

In the following corollary we consider for the  syzygy modules appearing in the preceding lemma,  Gr\"obner bases with respect to monomial orders induced by the given bases.

\begin{Corollary}
\label{liftback}
A Gr\"obner basis of $Z_i(C_\lpnt)$ is obtained by composing a Gr\"obner basis of
$Z_{i}(F_\lpnt)$ with the preimages of the elements of a Gr\"obner bases of $Z_{i-1}(G_\lpnt)$
with respect to  the epimorphism $Z_{i}(C_\lpnt)\to Z_{i-1}(G_\lpnt)$.
In particular if $F_{\lpnt}$ and $G_{\lpnt}$ have boundary Gr\"obner bases, then $C_{\lpnt}$
has boundary Gr\"obner bases.
\end{Corollary}

Now Corollary~\ref{liftback} yields

\begin{Corollary}
\label{niceiterated}
Let the resolution $F_\lpnt$ of $S/(u_1, \ldots, u_n)$ be an iterated
mapping cone by resolutions $G^{(j)}$ of (\ref{GroEqGj}) for $j = 1, \ldots, m-1$.
If each $G^{(j)}$ has boundary Gr\"obner bases, then $F_{\lpnt}$ has boundary Gr\"obner bases.
\end{Corollary}

As an application of these observations we obtain

\begin{Proposition}
\label{nice}
The Taylor complex and the iterated mapping cone  of an ideal with linear quotients have
boundary Gr\"obner bases. In particular, the Koszul complex attached to a regular sequence as
well as the Eliahou-Kervaire resolution for stable ideals  have boundary Gr\"obner bases.
\end{Proposition}

\begin{proof}
Let $F_\lpnt$ be the  Taylor complex on a sequence of monomials of length $m$. The complexes $G_\lpnt^{(j)}$ are Taylor complexes on sequences of length $m-1$. Thus by using induction on $m$, it  follows from Corollary~\ref{niceiterated} that $F_\lpnt$ has boundary Gr\"obner bases.
On the other hand, if $F_\lpnt$ is an iterated  mapping cone for an ideal with linear quotients, then all $G_\lpnt^{(j)}$ are Koszul complexes, which are special Taylor complexes, so that all $G^{(j)}_\lpnt$ have boundary Gr\"obner bases. Hence the desired result follows again by applying  Corollary~\ref{niceiterated}.
\end{proof}

To be more concrete let $T_\lpnt$ be the Taylor complex  attached with the sequence $u_1,\ldots,u_m$. For each $p$,   $T_p$ has the following basis: $e_F=e_{i_1}\wedge e_{i_2}\wedge \cdots \wedge e_{i_p}$ with $F=1\leq i_1<i_2<\cdots <i_p\leq m$, and the differential is given by
\[
\partial_p(e_F)= \sum_{j=1}^p(-1)^{j+1}\frac{u_F}{u_{F\setminus\{i_j\}}}e_{F\setminus\{i_j\}},
\]
where for any subset $G\subset [n]$ we let $u_G$ be the least common multiple of the monomials $u_i$ with $i\in G$. If we order the basis elements iteratively as described in Lemma~\ref{groebner}, then
$e_m > e_{m-1} > \cdots > e_1$ and more generally
$e_{i_1}\wedge e_{i_2}\wedge \cdots \wedge e_{i_p}> e_{j_1}\wedge e_{j_2}\wedge \cdots \wedge e_{j_p}$ if for some $k$ one has $i_p=j_p,\cdots,i_{k+1}=j_{k+1}$ and $i_k>j_k$. With this order,  the  elements $e_F$ with
$F\subset [n]$ and  $|F|=p$ form  boundary Gr\"obner bases.
Thus we obtain
\[
\ini(Z_p(T_\lpnt))=\Dirsum_{F\subset [m],\; |F|=p}I_Fe_F,
\]
with
\begin{eqnarray}
\label{initial}
I_F=(\frac{u_{F\union\{i\}}}{u_F} )_{i\in[m],\; i<\min(F)}.
\end{eqnarray}

\section{Stanley depth of syzygies}
\label{2}
In this section we consider lower bounds for the Stanley depth of syzygies.
First we give a lower bound in general for syzygies of $\ZZ^n$-graded submodules of free
modules. Then, in the case of squarefree modules we can give a considerably better bound.
The lower bounds have a form which is natural for syzygies. They essentially increase
by one for each successive syzygy. However the actual behavior of Stanley depth
of successive syzygy modules is probably far from the lower bound.

Our tool to obtain lower bounds for the Stanley depth is the following simple observation.

\begin{Lemma}
\label{sdepthm}
Let $F$ be a finitely generated $\ZZ^n$-graded $S$-module with multigraded basis ${\mathcal F}=e_1,\ldots,e_m$, $M$ a $\ZZ^n$-graded submodule of $F$, and
$\ini_{\mathcal F}(M)=\Dirsum_{j=1}^mI_je_j$. Then
\[
\sdepth  M\geq \min\{\sdepth I_1,\ldots,\sdepth I_m\}.
\]
\end{Lemma}

\begin{proof} Let $M=M_0\supset M_1\supset \cdots\supset M_r=0$ be any $\ZZ^n$-graded filtration of $M$.
Since, as we already observed,  for a short exact sequence of $\ZZ^n$-graded modules
\[ 0 \ra M^\prime \ra M \ra M^{\prime\prime} \ra 0 \]
one has $\sdepth M \geq \min \{ \sdepth M^\prime, \sdepth M^{\prime\prime} \}$, we deduce that
\[
\sdepth M\geq \max_i\{\sdepth M_i/M_{i+1}\}.
\]
Applying this general fact to the  filtration (\ref{filtration}) induced by $\mathcal F$, the result follows from (\ref{factors}).
\end{proof}

\medskip
Now we present our main results concerning Stanley depth of syzygies.

\begin{Theorem} \label{MainSyz}
Let $M$ be a $\ZZ^n$-graded submodule of a free module, and let $F_{\lpnt}$
be a free resolution as in (\ref{resolution}). 
Then for $p \geq 0$ the $p$'th syzygy module
$Z_p$ has Stanley depth greater than or equal to $p+1$, or it is 
a free module.
\end{Theorem}

\begin{proof}[Proof of Theorem \ref{MainSyz}]
Let $\Fc$ be a lex-refined basis for $F_p$. If $p \geq n$ then $Z_p$ is 
free, so suppose $p <  n$. 
By Theorem \ref{main},
$\ini_{\mathcal F}(Z_p)=\Dirsum_{j=1}^mI_je_j$, where the minimal set of monomial generators
of each of the monomial ideals $I_j$ belongs to $K[x_{p+1},\ldots,x_n]$. But then $\sdepth I_j \geq p+1$. In fact, Cimpoea\c s \cite[Corollary 1.5]{Ci2} showed that the Stanley depth of any  $\ZZ^n$-graded torsionfree $S$-module is at least 1. Hence the asserted  inequality for the Stanley depth of $I_j$ follows from \cite[Lemma 3.6]{He}. Now the desired inequalities for the Stanley depths 
of the syzygy modules follow from (\ref{sdepthm}).
\end{proof}

\begin{Remark}
In general the lower bound $(p+1)$ is probably far too small.
W.Bruns, C.Krattenthaler, and J.Uliczka consider in \cite{BrKr} syzygies of the Koszul
complex. They conjecture that the last half of these syzygies always have Stanley
depth equal to $n-1$.

On the other hand, the bound is sharp for the first syzygy module of any monomial ideal $I\subset S=K[x_1,x_2,x_3]$ with $\dim S/I=0$. Indeed, the predicted Stanley depth of $Z_1(I)$ is at least 2. It cannot be three, because otherwise $Z_1(I)$ would be free.
\end{Remark}

That indeed our lower bound for the Stanley depth is in general far too small can be seen in the following special case.

\begin{Proposition}
\label{regular}
Let $I\subset S$ be a monomial complete intersection minimally generated by $m$ elements, and $Z_p$ the $p$th syzygy module of $S/I$. Then
either $Z_p$ is free or 
\[
\sdepth Z_p\geq n-\lfloor \frac{m-p}{2}\rfloor.
\]
\end{Proposition}

\begin{proof} Let $u_1,\ldots,u_m$ be the regular sequence generating $I$. The Taylor complex associated with this sequence, which in this case is the Koszul complex, is a minimal free resolution of $S/I$. With the notation of (\ref{initial}) we have $I_F=(u_i)_{i\in[m],\; i<\min(F)}$, so that $\sdepth Z_p\geq \min\{\sdepth (u_i\:\; i<\min(F))\}$. By a result of Shen \cite{S} one has $\sdepth J=n-\lfloor m/2\rfloor$ for a monomial ideal $J\subset K[x_1,\ldots,x_n]$ generated by a regular sequence of length $m$. This yields the desired conclusion.
\end{proof}

In the case where $M$ is a squarefree ideal or more generally a squarefree module, we also
get better bounds. Recall that a $\ZZ^n$-graded $S$-module $M$ is
{\it squarefree} (defined by K.Yanagawa \cite{Ya}),
if it fulfils the following.
\begin{itemize}
\item  $M_{\bfa}$ is nonzero only if $\bfa \in \NN^n$.
\item When $\bfa$ is in $\NN^n$, with nonzero $i$'th coordinate, and $\epsilon_i$
is the $i$'th unit coordinate vector, the multiplication map
\[ M_\bfa \mto{\cdot x_i} M_{\bfa + \epsilon_i} \]
is an isomorphism of vector spaces.
\end{itemize}
Squarefree modules form an abelian category with squarefree projective covers.
In  particular kernels of morphisms of squarefree modules are squarefree, and so syzygies
modules in a squarefree resolution of a squarefree module, are squarefree.

\begin{Lemma}
Let $M$ be a squarefree submodule of a free module $F$. Then for any
term ordering on $F$, the initial module $\ini(M)$ is a squarefree module.
\end{Lemma}

\begin{proof} Let $g_1, \ldots, g_p$ be a basis for $M_\bfa$, such that the
$\ini(g_r) = u_r e_{i_r}$ are a basis for $\ini(M)_{\bfa}$, and suppose $a_i \neq 0$.
Then $x_i g_1, \ldots, x_i g_p$ are a basis for $M_{\bfa + \epsilon_i}$, and their
initial terms are the $x_i u_r e_{i_r}$ which form a basis for
$\ini(M_{\bfa + \epsilon_i})$.
\end{proof}

In the last section we show that for squarefree ideals there is a considerably better lower
bound for the Stanley depth than $1$, which we use in Theorem \ref{MainSyz}.
Using this we get the following.

\begin{Theorem}
\label{squarefree}
 Let $M$ be a squarefree submodule of a free module, with $d+1$ the  smallest
degree of a generator of $M$. Let $s$ be the largest integer such that
$(2s+1)(s+1)  \leq n+1-d-p$. Then for $p \geq 1$
the $p$'th  syzygy module in a squarefree
resolution of $M$ is either free, or it has Stanley depth
greater or equal to $2s+1+d+p$. 
Explicitly this is
\[ 2 \left \lfloor  \frac{\sqrt{2n-2d-2p+2.25}+0.5}{2} \right \rceil +d+p-1. \]
\end{Theorem}

\begin{proof} Use a term order as in (\ref{order}) on the $p$'th term $F_p$ 
in the resolution. We get
\begin{equation} \label{StaEqIe} \ini(Z_p) = \Dirsum_{j=1}^mI_je_j,
\end{equation}
where $I_j$ is a squarefree monomial ideal.
Since $\ini(Z_p)$ is a squarefree module, each generator $g_{ji}e_j$ of
$I_j e_j$ is a squarefree term.

Suppose first that the resolution is minimal. Then
the total degree of $e_j$ is
at least $d+p$ where $d+p \leq n$. 
Hence the generators of $I_j$ involves no more than $n-d-p$ variables,
corresponding to the coordinates of the multidegree of $e_j$ which are zero.
The result then follows from Lemma~\ref{sdepthm}  and Theorem~\ref{SqfreeStDe}.

In the case that the resolution 
is not necessarily minimal, the syzygies $Z^\prime_p$ of such a resolution 
differ from the syzygies $Z_p$ of the minimal free resolution by a free
summand, that is, $Z^\prime_p=Z_p\oplus F$, so either 
$Z^\prime_p$ is free or it has Stanley depth greater or equal to 
the Stanley depth of $Z_p$.
\end{proof}



\section{Stanley depth of squarefree monomial ideals}
\label{4}

In this section we show that the Stanley depth of any squarefree monomial ideal
in $n$ variables, is bounded below by a bound of order  $\sqrt{2n}$.
This is quite in contrast to ordinary
depth where for instance the maximal ideal $(x_1, \ldots, x_n)$ has depth one.

Our argument is based on a construction
of interval partitions \cite{Ke} by M.Keller et.\ al.\ which is further refined
M.Ge, J.Lin, and Y.-H.Shen in \cite{Lin}. The argument is an application of
Proposition 3.5 in \cite{Lin}.

\medskip
We recall the construction of \cite{Ke}. Let $[n] = \{1,2, \ldots, n\}$. For subsets
$A$ and $B$ of $[n]$, the interval $[A,B]$ consists of the subsets $C$ of $[n]$ such
that $A \sus C \sus B$. We think of the elements of $[n]$ arranged clockwise around
the circle and for $i,j$ in $[n]$ let the block $[i,j]$ be the set of points starting
with $i$, going clockwise, and ending with $j$. Now given $A \sus [n]$, and a
real number $\delta \geq 1$, called a {\it density}, the {\it block structure
of $A$} with respect to $\delta$ is a partition of the elements of $[n]$ into
connected blocks $B_1, G_1, B_2, G_2, \ldots, B_p, G_p$ fulfilling the following.
\begin{itemize}
\item The first (going clockwise) element of $b_i$ of $B_i$ is in $A$.
\item Each $G_i$ is disjoint from $A$.
\item For each $B_i$ we have
\[\delta \cdot |A \sect B_i| - 1 < |B_i| \leq \delta \cdot
|A \sect B_i|.\]
\item For each $y$ such that $[b_i,y]$ is a proper subset of $B_i$, we have
\[|[b_i, y]| + 1 \leq \delta \cdot |[b_i,y] \sect A|.\]
\end{itemize}
For $1 \leq \delta \leq \frac{(n-1)}{|A|}$ the block structure for a subset $A$ exists and
is unique by Lemma 2.7 in \cite {Ke}.
Let ${\mathscr G}_\delta$ be the union $G_1 \union \cdots \union G_p$.
We then define the set $f_\delta(A)$ to be $A \union {\mathscr G}_\delta(A)$. The
intervals we shall study will now be of the form $[A,f_\delta(A)]$ or closely related.

For certain values of $n$ and cardinalities of $A$, the intervals fulfil
some very nice properties. The following are basic facts from \cite{Ke}. It is a
synopsis of Lemma 3.1, Lemma 3.2 and Lemma 3.5 there.
 \begin{Lemma} Let  $n = as + a + s$.
\begin{itemize}
\item[1.] If $A \sus [n]$ is an $a$-set, then $f_{s+1}(A)$ is and $(a+s)$-set.
\item[2.] The intervals $[A,f_{s+1}(A)]$ are disjoint when $A$ varies over the $a$-sets.
\end{itemize}
\end{Lemma}

We are interested in getting disjoint intervals, but we need a way to adopt the above
lemma to the case of arbitrary $n$, and to be able to vary $a$ and $s$. The following
still fixes $s$ and $a$ but allows $n$ to be arbitrary above a bound. It is Proposition 3.3 in
\cite{Lin}.

\begin{Proposition} Let $n \geq as + a + s$ and $A \sus [n]$ an $a$-set.
Consider $\tilde{A} = A \union \{n+1, \ldots, n+n-a\}$ as a subset of
$[ns + n + s]$.
\begin{itemize}
\item[1.] The intersection $f_{s+1}(\tilde{A}) \sect [n]$ is an $(a+s)$-set.
\item[2.] The intervals $[A, f_{s+1}(\tilde{A}) \sect [n]]$ are disjoint as $A$ varies
over the $a$-sets.
\end{itemize}
\end{Proposition}

Finally we need to be more flexible with $a$ and $s$, and still have disjoint intervals.
The following is Proposition 3.5 of \cite{Lin} specialized to the case when $d=1, d+q = a$
and $d+l = b$.
\begin{Proposition} \label{4disjoint}
Let $A$ and $B$ be subsets of $[n]$ of cardinalities $a \leq b$. 
Suppose $s^\prime \leq s$ are non-negative integers such that
\[ n+1 \geq (b+1)(s^\prime + 1) \geq (a+1)(s+1). \]
Consider $\tilde{A}$ as a subset of $[ns+n+s]$ and $\tilde{B}$ as a subset of
$[ns^\prime + n + s^\prime]$. Then if $B$ is not in $[A, f_{s+1}(\tilde{A}) \sect[n]]$,
this  interval is disjoint from the interval $[B, f_{s^\prime + 1}(\tilde{B}) \sect [n]]$.
\end{Proposition}

We are now ready to prove our theorem.

\begin{Theorem}\label{SqfreeStDe}
Let $s$ be the largest integer such that $n+1 \geq (2s+1)(s+1)$.
Then the Stanley depth of any squarefree monomial ideal in $n$ variables is
greater or equal to $2s+1$. Explicitly this lower bound is
\[ 2 \left \lfloor \frac{\sqrt{2n+ 2.25} + 0.5}{2} \right \rfloor - 1. \]
\end{Theorem}

\begin{Remark} For $n = 5$, the above bound says that the Stanley depth is greater than
or equal to $3$ which is best possible, since this is the Stanley depth of the
maximal ideal $(x_1,\ldots, x_5)$.
\end{Remark}

\begin{Remark} In \cite{Ci2} it is shown that the Stanley depth of the
squarefree Veronese ideal generated by squarefree monomials of degree $d$ has
Stanley depth less or equal to
$\lfloor \frac{n+1}{d+1} \rfloor + d - 1$. With $d+1$ approximately $\sqrt{n+1}$
this is approximately $2\sqrt{n+1} - 2$. Thus our lower bound is right up to a
constant.
\end{Remark}

\begin{proof}[Proof of Theorem \ref{SqfreeStDe}]
The squarefree ideal $I$ corresponds
to an order filter $P_I$ of the poset consisting of subsets
of $[n]$, by taking supports of the squarefree monomials in $I$.
By J.Herzog et.al \cite{He}, if we have a partition $\Pscr$ of $P_I$,
the Stanley depth of $I$ is greater or equal to the minimum cardinality
of any subset $B$ of $[n]$ such that $[A,B]$ is an interval in $\Pscr$.
We shall therefore construct a suitable partition to give the lower bound.

Given positive integers $n,r$ and $s$ with $r >s$ and
\[(n+1) \geq (r+1)(s+1).\]
Define the sequence $\sigma : [r] \rightarrow \ZZ$ by
\[ \sigma(i) = \begin{cases} s, & i \geq s+1 \\
                               s+k, & i = s+1 -k \leq s+1
                 \end{cases}.
\]

Note that if $u \leq v$ then
$(u+1)v \geq u (v+1)$.
Therefore the expression $(i+1)(\sigma(i) + 1)$ weakly decreases as $i$ decreases,
enabling us to apply Proposition \ref{4disjoint}.

\medskip
 We now construct  a partitition ${\Pscr}$ of $P_I$ as follows.
Let $\Pscr_1$ consist of all intervals
\[ [\{i\}, f_{\sigma(1) + 1}(\widetilde{\{i\}}) \sect [n]] \]
where $\{i\}$ is in $P_I$. Let $\Pscr_2$ consist of all intervals
\[ [\{i_1,i_2\}, f_{\sigma(2) + 1}(\widetilde{\{i_1, i_2\}}) \sect [n]] \]
where $\{i_1, i_2\}$ is in $P_I$ but not in any of the intervals in $\Pscr_1$.
By Proposition \ref{4disjoint}, $\Pscr_2$ will consist of disjoint intervals. Having
constructed $\Pscr_{a-1}$ we construct $\Pscr_a$ by adding to
$\Pscr_{a-1}$ all intervals $[A, f_{\sigma(a) + 1}(\tilde{A}) \sect [n]]$
where $A$ is an $a$-set in $P_I$ not in any
interval of $\Pscr_{a-1}$. Having reached $\Pscr_r$ we obtain
$\Pscr$ by adding all trivial intervals $[B,B]$ where $B$ is in $P_I$ but not
in any of the intervals in $\Pscr_r$. Note that each such $B$ has cardinality
greater or equal to $r+1$.

The Stanley depth of the partition $\Pscr$ will be the smallest of the numbers
$r+1$ and $i+\sigma(i)$ for $i = 1, \ldots, r$, which is the minimum of $r+1$ and $2s+1$.
Choose $r = 2s$ and $s$ to be the largest number such that $n+1 \geq (2s+1)(s+1)$.
Then the Stanley depth of $\Pscr$ is $2s+1$.
We get
\begin{eqnarray*}
(2s+1)(s+1) & \leq & n+1 \\
4s^2 + 6s + 2 & \leq & 2n+2 \\
(2s+1.5)^2 & \leq & {2n + 2.25} \\
2s+2 & \leq & \sqrt{2n+2.25} + 0.5
\end{eqnarray*}
which gives
\begin{eqnarray*} s &\leq &\left \lfloor \frac{\sqrt{2n+2.25}+0.5}{2} \right\rfloor - 1 \\
2s+1 & \leq & 2 \left\lfloor \frac{\sqrt{2n+2.25}+0.5}{2} \right\rfloor - 1.
\end{eqnarray*}
The largest value of $2s+1$ is then given by the right side above.
\end{proof}

{}

\end{document}